\documentclass{amsart}

\usepackage{amsfonts,amssymb,amsmath,amsthm,latexsym,booktabs,lscape,color,enumerate, amscd}
\usepackage{marginnote}

\usepackage{calrsfs}
\DeclareMathAlphabet{\pazocal}{OMS}{zplm}{m}{n}

\usepackage{tikz-cd}
\input xygraph
\theoremstyle{definition}
\newtheorem{definition}{Definition}[section]
\newtheorem{remark}[definition]{Remark}
\newtheorem{example}[definition]{Example}

\theoremstyle{plain}
\newtheorem{lemma}[definition]{Lemma}
\newtheorem{proposition}[definition]{Proposition}
\newtheorem{theorem}[definition]{Theorem}
\newtheorem{corollary}[definition]{Corollary}

\newtheorem{notation}[definition]{Notation}

\def\G{\mathcal G}
\def\g{\mathfrak{g}}
\def\h{\mathfrak{h}}
\def\affF{\mathop{\hbox{\bf F}}}
\def\affZ{\mathop{\hbox{\bf Z}}}
\def\affG{\mathop{\hbox{\bf G}}}
\def\affH{\mathop{\hbox{\bf H}}}
\def\affM{\mathop{\hbox{\bf M}}}
\def\affGL{\mathop{\hbox{\bf GL}}}
\def\lie{\mathop{\mathrm{Lie}}}
\def\op#1{{#1}^{\hbox{\tiny o}}}
\newcommand{\Z}{\mathrm{Z}}
\newcommand{\M}{\mathrm{M}}
\newcommand{\T}{\mathcal{T}}

\newcommand{\A}{\mathcal{A}}
\newcommand{\Bc}{\mathcal{B}}
\newcommand{\Mm}{\mathcal{M}}

\newcommand{\V}{\mathcal{V}}
\newcommand{\Tri}{\mathrm{Trian \,}}
\newcommand{\Id}{\mathrm{Id}}

\newcommand{\Imm}{\mathrm{Im}}
\newcommand{\aut}{\mathrm{Aut}}
\newcommand{\End}{\mathrm{End}}
\newcommand{\GL}{\mathrm{GL}}
\newcommand{\gl}{\mathrm{gl}}
\newcommand{\terder}{\mathrm{TerDer}}

\newcommand{\teraut}{\mathrm{TerAut}}
\newcommand{\innteraut}{\mathrm{InnTerAut}}

\newcommand{\der}{\mathrm{Der}}
\newcommand{\innder}{\mathrm{Innder}}
\newcommand{\outder}{\mathrm{Outder}}
\def\innaut{\mathop{\hbox{Innaut}}}
\newcommand{\ad}{\mathrm{ad}}
\newcommand{\inn}{\mathrm{In}}
\def\alg{\mathop{\hbox{\bf alg}}}
\def\grp{\hbox{\bf Grp}}
\def\set{\hbox{\bf Set}}

\def\affteraut{\mathop{\hbox{\bf TerAut}}}
\def\affaut{\mathop{\hbox{\bf Aut}}}
\def\affinnaut{\mathop{\hbox{\bf InnAut}}}
\def\affinnteraut{\mathop{\hbox{\bf InnTerAut}}}

\def\e{\varepsilon}
\def\a{\alpha}
\def\b{\beta}
\usepackage{color}



\begin{document}

\title[Ternary mappings of triangular algebras]{Ternary mappings of triangular algebras}

\author[Mart\'in Barquero, Mart\'in Gonz\'alez, S\'anchez-Ortega, Vandeyar]{Dolores Mart\'in Barquero$^{1}$, C\'andido Mart\'in Gonz\'alez$^{2}$, Juana S\'anchez-Ortega$^{3}$, Morgan Vandeyar$^{3}$}
\address[1]{Departamento de Matem\'atica Aplicada. Escuela de Ingenier\'\i as Industriales. Universidad de M\'alaga, Campus de Teatinos. 29071 M\'alaga,   Spain.}
\address[2]{Departamento de \'Algebra, Geometr\'ia y Topolog\'ia. Facultad de Ciencias. Universidad de M\'alaga, Campus de Teatinos. 29071 M\'alaga, Spain.}
\address[3]{Department of Mathematics and Applied Mathematics. University of Cape Town \\ Cape Town, South Africa.}
\email{dmartin@uma.es}
\email{candido@apncs.cie.uma.es}
\email{juana.sanchez-ortega@uct.ac.za}
\email{VYRMOR001@myuct.ac.za}

\maketitle


\begin{abstract}
We take a categorical approach to describe ternary derivations and ternary automorphisms of triangular algebras. New classes of automorphisms and derivations of triangular algebras are also introduced and studied. 
\end{abstract}


\section{Introduction}

Let $\A$ be an algebra, a triple of linear maps $(d_1, d_2, d_3)$ of $\A$ is a {\bf ternary derivation} of $\A$ if 
\[
d_1(xy) = d_2(x)y + xd_3(y),
\]
for all $x, y \in \A$. Clearly, if $d_1 = d_2 = d_3$, then $d_1$ is a derivation of $\A$. Perhaps a more interesting feature is the following: if  $d_1 = d_2$, then $d_1$ is a generalized derivation of $\A$. Generalized derivations were introduced by Bre\v sar \cite{B1} and further studied by Hvala \cite{H1, H2}.

The notion of a ternary derivation dates back to the Principle of Local Triality, which establishes that every element in 
\[
\mathfrak{o}(C, q) = \big \{d \in \End_F(C) | \, \, (d(x), y) + (x, d(y)) =0 \big \},
\] 
where $(C, (\cdot, \cdot))$ is a (generalized) octonion algebra, gives rise to a ternary derivation of $C$. See \cite[Theorem 3.31]{Schafer}, \cite[Section 1]{AlbertoTriality} and \cite{Chema} for details. The terminology of ternary derivations was introduced by Jimen\'ez-Gestal and P\'erez-Izquierdo \cite{JP} in 2003. They described ternary derivations of the generalized Cayley-Dickson algebras over a field of characteristic not 2 and 3. More recently, Shestakov \cite{Sh1, Sh2} described ternary derivations of separable associative and Jordan algebras and of Jordan superalgebras; see also \cite{ZS}. Leger and Luks \cite{LL} also studied ternary derivations of Lie algebras, but referred only to the component $d_2,$ calling it a generalized derivation (different from Bre\v{s}ar's definition).

Here, we focus our attention on ternary derivations and ternary automorphisms of triangular algebras, which were introduced by Chase \cite{Cha} in the early 1960s. Our motivation comes from the large number of papers devoted to the study of different maps of triangular algebras \cite{Bk1, Bk2, BM, Ch2, ChLie, CM, LB}. We use some notions from category theory and group functors, for which we refer the reader to \cite{Maclane, Pareigis, Waterhouse}.

The paper is organized as follows: in Section 2, we provide the required background on triangular algebras, on category theory and on group functors. Section 3 deals with ternary automorphisms, in this section   Theorem \ref{Candidothm1} states that every ternary automorphism of an arbitrary algebra $\A$ is on the form $(R_y L_x \sigma, L_x \sigma, R_y \sigma)$, for $\sigma$ an automorphism of $\A$ and $x, y$ invertible elements of $\A$; a description of ternary derivations of $\A$ is derived. Section 4 focuses on ternary derivations of triangular algebras; our results from Section 3 allow us to provide a precise description of them in Theorem \ref{terderchar}; the relationship between the components of a ternary derivation is studied (see Theorem \ref{components}), and a characterization for a ternary derivation to be inner is given in Theorem \ref{innerterder}. The last section deals with a new class of automorphisms of triangular algebras.


\section{Preliminaries}

Throughout the paper we consider unital associative algebras over a fixed commutative unital ring of scalars $R$.

Let $\Mm$ be an $R$-module and $\A$ a commutative, associative and unital $R$-algebra. We write $\Mm_\A$ to denote the $R$-module scalar extensions $\Mm \otimes_R \A$ of $\Mm$. The $R$-module $\hom_R(\Mm,\Mm_\A)$ is an $\A$-module with the obvious addition and product 
$af$, where $a \in \A$ and $f \in \hom_R(\Mm, \Mm_\A)$. There is an isomorphism of $\A$-modules $\theta_\A\colon \hom_R(\Mm, \Mm_\A)\to \hbox{End}_\A(\Mm_\A)$ given by $\theta_\A(f)(m\otimes a) = af(m)$ for any $a\in \A$ and $m\in \Mm$. As a consequence, if  $\a\colon \A\to \Bc$ is a homomorphism of $R$-algebras (where $\A$ and $\Bc$ are both commutative), then for any $T\in\text{End}_\A(\Mm\otimes \A)$ we can construct 
$S\in\text{End}_\Bc(\Mm\otimes \Bc)$ making commutative the following diagram:
\begin{equation}\label{bolo} 
\xygraph{ !{<0cm,0cm>;<1.5cm,0cm>:<0cm,1.2cm>::} 
!{(0,0) }*+{\Mm\otimes \A}="a"
!{(1.5,0) }*+{\Mm\otimes \A}="b" 
!{(0,-1.5)}*+{\Mm\otimes \Bc}="c" 
!{(1.5,-1.5)}*+{\Mm\otimes \Bc.}="d"
!{(.7,.2)}*+{_{T}}
"a":"b"  "a":_{\mathrm{Id}\otimes\a}"c"
"b":^{\mathrm{Id}\otimes\a}"d"  "c":_{S}"d"}
\end{equation}
In fact, let $i\colon \Mm\to \Mm \otimes \A$ given by $i(m) = m\otimes 1_\A$ and  consider the map $(\mathrm{Id}\otimes\a)Ti\colon \Mm\to \Mm\otimes \Bc$. Then $S:=\theta_\Bc((\mathrm{Id}\otimes\a)Ti)$ gives \begin{equation*}S(m\otimes 1_\Bc)=(\mathrm{Id}\otimes\a)T(m\otimes 1_\A), \end{equation*}
as desired.

\subsection{Triangular algebras} \label{preli}

Let $A$ and $B$ be algebras and $M$ a nonzero $(A, B)$-bimodule. The following set becomes a unital, associative algebra under the usual matrix operations:
\[
\Tri(A, M, B) = \left(
\begin{array}{@{}cc@{}}
A & M  \\
  & B
\end{array}
\right) =
\left\{
\left(
\begin{array}{@{}cc@{}}
a & m  \\
  & b
\end{array}
\right)\middle|  \, \, a \in A, m\in M, b \in B
\right\}.
\]
An algebra $\T$ is called a {\bf triangular algebra} if there exist algebras $A$, $B$ and a nonzero $(A, B)$-bimodule $M$ such that $\T$ is isomorphic to $\Tri(A, M, B)$. Examples of triangular algebras are (block) upper triangular matrix algebras, triangular Banach algebras and nest algebras. 

\begin{notation}\label{notation}
{\rm
Let $\T = \Tri(A, M, B)$ be a triangular algebra.
We write $1_A$, $1_B$ to denote the units of the algebras $A$, $B$, respectively. The unit of $\T$ is the element
$
1 := \left(
\begin{array}{@{}c@{\; \;}c@{}}
1_A & 0  \\
    & 1_B
\end{array}
\right).
$
The elements $p :=
\left(
\begin{array}{@{}c@{\; \;}c@{}}
1_A & 0  \\
    & 0
\end{array}
\right)$ and $q := 1 - p$ are orthogonal idempotents of $\T$.

We consider the Peirce decomposition $\T = p\T p \oplus p \T q \oplus q \T q$ of $\T$ associated to $p$. Note that $p\T p$, $q \T q$ are subalgebras of $\T$ isomorphic to $A$, $B$, respectively; while $p \T q$ is a $(p \T p, q \T q)$-bimodule isomorphic to $M$. 
}
\end{notation}

\subsection{Groups and set functors.} Let $\alg_R$ denote the category of associative commutative and unital $R$-algebras, $\set$ the category of sets and $\grp$ the category of groups. We call a functor from $\alg_R$ to $\set$ (respectively, $\grp$) a \textbf{set functor} (respectively, a \textbf{group functor}).


We write $\GL(\Mm)$ to denote the group of all invertible $R$-endomorphisms of an $R$-module $\Mm$. The \textbf{general linear group functor} $\affGL(\Mm)$ (see \cite[Examples 2.2]{Jantzen}, \cite[p. 186]{DG} or \cite[Section 1.8]{Milne} under certain special finiteness conditions) \textbf{associated to $\Mm$} is defined as 
\[
\affGL(\Mm)(\A) := \GL(\Mm_\A), 
\]
for all $\A \in \alg_R$, and for any $f\in\hom_{\alg_R}(\A,\Bc)$
we define 
\[
\affGL(\Mm)(f):= \hat f\colon \GL(\Mm_\A)\to\GL(\Mm_\Bc), 
\]
in the following way: given $T\in\GL(\Mm_\A)$, an application of \eqref{bolo} guarantees the existence of $S\in\hbox{End}_{\Bc}(\Mm)$ such that 
$S (\mathrm{Id}\otimes f)=(\mathrm{Id}\otimes f)T$.
It is then straightforward to check that  $S\in\GL(\Mm_{\Bc})$, so we can then define $\hat f(T)=S$. 
Consequently, we can write
$\hat f(T)(m\otimes 1_\Bc) = \big(\mathrm{Id} \otimes f \big)T(m\otimes 1_\A)$, for all $T \in \GL(\Mm_\A)$. 

For any $R$-module $\Mm$, the \textbf{set functor} $\affM$ \textbf{associated to} $\Mm$ is defined as $\affM(\A) := \Mm_\A$ and $\affM(f) := \mathrm{Id} \otimes f$, 
for all $\A \in \alg_R$ and $f \in \hom_{\alg_R}(\A,\Bc)$, where $\A, \Bc \in \alg_R$.

As a particular case of the general linear group functor, we can define a group functor
$\affGL_n(R)$ by $\affGL_n(R)(\A):= GL_n(\A)$, where $GL_n(\A)$ stands for the group of invertible $n\times n$ matrices with entries in $\A$, 
and $\affGL_n(R)(\a)(x_{ij}) = (\a(x_{ij}))$ for any $\a\in\hom_{\alg_R}(\A,\Bc)$. If $\Mm$ is a finite-dimensional $K$-vector space, then $\affGL(\Mm)\cong\affGL_n(K)$, where $n = \dim(\Mm)$.

\smallskip 

We write $R[\varepsilon]$ for the algebra of dual numbers over $R$, that is, $R[\varepsilon]:= R\oplus \varepsilon R$, where $\varepsilon^2=0$. Clearly, the map $\pi: R[\varepsilon] \to R$ given by $\pi(x + \varepsilon y) = x$ is an epimorphism of $R$-algebras. Given an $R$-module $\Mm$, the {\bf tangent $R$-module of $\affM$ at $n \in \Mm$} (see \cite[p. 257]{EH}) is the following $R$-module of $\affM(R[\varepsilon])$: 
\[
T_n(\affM) = \Big \{ m \in \Mm  \mid \, n + \varepsilon m \in \affM(R[\varepsilon]) \Big \}.
\]
It is straightforward to show that $\affM(R[\varepsilon]) = \Mm \oplus \varepsilon \Mm$, which gives $T_n(\affM) = \Mm$.

Let $\Mm$ be an $R$-module and $\affF$ a subfunctor of $\affGL(\Mm)$. The kernel of the
group homomorphism $\affF(\pi)$ is a Lie algebra, denoted $\lie(\affF)$. We then have a short exact sequence 
\[
0 \rightarrow \lie(\affF)\rightarrow\affF(R[\varepsilon])\buildrel {\affF(\pi)}\over{\rightarrow}\affF(R) \rightarrow 0,
\]
which is split since the inclusion map $\iota: R \to R[\varepsilon]$ induces $\affF(\iota)\colon \affF(R)\to\affF(R[\varepsilon])$ such that 
$\affF(\pi)\affF(\iota) = \mathrm{Id}$. Therefore, 
$\affF(R[\varepsilon])\cong \lie(\affF)\times \affF(R)$. 

Consider now $\affF$, a subfunctor of $\affGL(\V)$
for an $R$-module $\V$ and $\affM$ the set functor associated to an $R$-module $\Mm$, as before. If $f\colon\affF\to\affM$ is a natural transformation
and $p = f_R(1)$, then there is a map in $\set$ (so not necessarily linear)
$df_1\colon\lie(\affF)\to T_p(\M)$ such that
$f_{R[\e]}(1+\e x)=p+\e df_1(x)$ for any $x\in\lie(\affF)$.
In fact, the commutativity of the square 
\begin{center}
\begin{tikzcd}
\affF(R[\e])\arrow[d,"f_{R[\e]}"']\arrow[r, "\affF(\pi)"] & \affF(R)\arrow[d,"f_R"] \\ \affM(R[\e])\arrow[r,"\affM(\pi)"']& \affM(R).
\end{tikzcd}
\end{center}
gives that $x\in\lie(\affF)$ if and only if $1+\e x\in \affF(R[\e])$ and then 
$\affF(\pi)(1+\e x)=1$. Thus:
$$
p = f_R(1) = (f_R \affF(\pi))(1 + \e x) = \affM(\pi)\big(f_{R[\e]}(1 + \e x)\big),
$$
which implies that $f_{R[\e]}(1+\e x)=p+\e df_1(x)$, where 
$df_1\colon\lie(\affF) \to \affM(R) = \Mm$. 
\smallskip

It is well known that pullbacks exist in $\set$. More precisely, given $G, H \in \grp$, $Z \in \set$ and $f: G \to Z$ and $g: H \to Z$ morphisms in $\set$, the pullback of the pair $(f, g)$, denoted $G\times_Z H$, is the set $\{(x, y)| \, \,  x \in G, y \in H, f(x) = g(y)\}$. We then have a pullback square (where $\pi_1$ and $\pi_2$ refer to the usual projections): 
\begin{equation*}
\xygraph{ !{<0cm,0cm>;<1.5cm,0cm>:<0cm,1.2cm>::} 
!{(0,0) }*+{G\times_Z H}="a"
!{(1,0) }*+{H}="b" 
!{(0,-1)}*+{G}="c" 
!{(1,-1)}*+{Z}="d"
!{(.2,-.2)}*+{\lrcorner}
!{(.7,.2)}*+{_{\pi_2}}
"a":"b"  "a":_{\pi_1}"c"
"b":^{g}"d"  "c":_{f}"d"}
\end{equation*}
A sufficient condition for $G\times_Z H$ to be a group is given in the next result, which proof is straightforward and therefore omitted.

\begin{lemma} \label{req} 
Let $G, H \in \grp$, $Z \in \set$ and  $f\colon G\to Z$ and $g\colon H\to Z$ morphisms in $\set$. Assume that $f$ and $g$ satisfy the following conditions: 
\begin{enumerate}
\item[\rm (i)] $f(1_G) = g(1_H)$.
\item[\rm (ii)] $f(x) = g(y) and f(x') = g(y')$ imply $f(xx')=g(yy')$ for any $(x, y), (x', y') \in G \times H$.
\item[\rm (iii)] $f(x)=g(y)$  implies $f(x^{-1})=g(y^{-1})$ for any $(x, y) \in G \times H$.
\end{enumerate}
Then $G\times_Z H \in \grp$ with the componentwise product.
\end{lemma}

\begin{remark}\label{grgr}
A version of Lemma \ref{req} for Lie algebras follows: let $\g,\h$ be Lie algebras over $R$ and $Z$ an $R$-module. Suppose that the maps $\a\colon\h\to Z$ and $\b\colon\g\to Z$ satisfy: 
\begin{enumerate}
\item[\rm (i)'] $\a(0) = \b(0)$.
\item[\rm (ii)'] $\a(x) = \b(y)$ and $\a(x') = \b(y')$ imply $\a(x + x') = \b(y + y')$ and $\a([x, x'])=\b([y, y'])$, 
for any $(x, y), (x', y') \in \g \times \h$.
\item[\rm (iii)'] $\a(x)=\b(y)$ implies $\a(-x)=\b(-y)$, for any $(x, y) \in \g \times \h$.
\end{enumerate}
Then $\g\times_Z\h$ is a Lie algebra endowed with the componentwise operations.
\end{remark}

\begin{definition} \label{functorpullback}
Let $\affG$, $\affH$ be group functors and $\affZ$ the set functor associated to an $R$-module $Z$. Assume that $\psi\colon \affH\to\affZ$ and $\varphi\colon\affG\to\affZ$ are natural transformations. Then, we define a new set functor $\affG \times_{\mathbf{Z}} \affH$ as 
\[
\big(\affG \times_{\mathbf{Z}} \affH \big) (\A) := \affG(\A) \times_{\mathbf{Z}(\A)} \affH(\A),
\]
for any $\A \in \alg_R$, and for any $\xi\in\hom_{\alg_R}(\A,\Bc)$
we define 
\[
\big(\affG \times_{\mathbf{Z}} \affH \big)(\xi)(a) := \big(\affG(\xi)(a), \affH(\xi)(a) \big), \quad \forall \, a \in \A. 
\]
The fact that $\affG\times_{\mathbf{Z}}\affH$ is well defined on arrows uses the naturality of $\psi$ and $\varphi$, as well as, the fact that $\affG(\A)\times_{\mathbf{Z}(\A)}\times\affH(\A)$ is a pullback.
\end{definition}

\begin{proposition} \label{squareprop}
Let $\affG$, $\affH$ be group functors and $\affZ$ the set functor associated to an $R$-module $Z$. Suppose that $\varphi: \affG\to \affZ$ and $\psi: \affH \to \affZ$ are natural transformations such that 
$\varphi_\A: \affG(\A) \to \affZ(\A)$, 
$\psi_\A: \affH(\A) \to \affZ(\A)$ satisfy Lemma \ref{req} conditions for all $\A \in \alg_R$. Then $\affG \times_{\mathbf{Z}} \affH$ is a group functor with Lie algebra 
$\lie(\affG) \times_{\mathbf{Z}} \lie(\affH)$. In other words, $\lie(\affG) \times_{\mathbf{Z}} \lie(\affH)$ is the pullback of the square
\begin{equation*}
\xygraph{ !{<0cm,0cm>;<1.5cm,0cm>:<0cm,1.2cm>::} 
!{(0,0) }*+{\lie(\affG\times_{\mathbf{Z}}\affH)}="a"
!{(1.7,0) }*+{\lie(\affH)}="b" 
!{(0,-1)}*+{\lie(\affG)}="c" 
!{(1.7,-1)}*+{Z}="d"
!{(.2,-.3)}*+{\lrcorner}
!{(1,.2)}*+ {}
"a":"b"  "a":"c"
"b":^{d\psi}"d"  "c":_{d\varphi}"d"},
\end{equation*}
where $d\varphi: \lie(\affG) \to Z$ and $d\psi:\lie(\affH) \to Z $ refer to the differential maps at $1 \in \affG(R)$ and $1 \in \affH(R)$, respectively.
\end{proposition}

\begin{proof}
The result follows from Lemma \ref{req} and the considerations above. In fact, we check here that $d\varphi$ and $d\psi$ satisfy Remark \ref{grgr} (ii)'. Assume that $d\varphi(x) = d\psi(y)$ and $d\varphi(x') = d\psi(y')$, for any $x, x'\in\lie(\affG)$ and $y, y'\in\lie(\affH)$. From $1 + \e x, 1 + \e x'\in \affG(R[\e])$ we get that 
\begin{align*}
& 1+(1\otimes\e) x+(\e\otimes 1) x'+(\e\otimes\e) x x'\in G(R[\e]\otimes R[\e]), \mbox{ and}
\\
& 1-(1\otimes\e) x-(\e\otimes 1) x'+(\e\otimes\e) x x'\in G(R[\e]\otimes R[\e]),
\end{align*}
which gives $1 + \delta[x,x']\in G(R[\delta])$, where $\delta:= \e\otimes\e$. Similarly, $1+ \delta[y,y']\in H(R[\delta])$. On the other hand, we can write 
$$
1 + \delta [x, x'] = \{1 + \e x, 1 + \e x'\}, \qquad
1 + \delta [y, y'] = \{1 + \e y, 1 + \e y'\}
$$
where $\{g_1, g_2\} = g_1g_2g_1^{-1}g_2^{-1}$ is the commutator in the corresponding group. Then 
$$
\varphi(1 + \delta[x, x']) = \varphi(\{1 + \e x, 1 + \e x'\})=
\psi(\{1+\e y,1+\e y'\} = \psi(1+\delta[y,y']),
$$
which implies that 
$$\varphi(1) + \delta d\varphi([x,x'])=
\psi(1) + \delta d\psi([y,y']),$$
and (ii)' clearly follows. 
\end{proof}

%

\section{Ternary automorphisms of triangular algebras}

Ternary automorphisms and their infinitesimal counterpart, ternary derivations, were investigated
in the context of algebras obtained by the Cayley-Dickson process in \cite{JP}. 
However, the notion of a ternary automorphism was introduced much earlier by Petersson \cite{Petersson} (under the name of Albert autotopy) in the context of alternative algebras.

In this section, we develop the notions of ternary Lie automorphisms and introduce inner ternary automorphisms.

Let $\A$ be an algebra. A triple $(\sigma_1, \sigma_2, \sigma_3)$, where $\sigma_i \in \GL(\A)$ for all $i$, is said to be a {\bf ternary automorphism} of $\A$ if it satisfies one of the following equivalent conditions:
\begin{align*}
\sigma_1(xy) &= \sigma_2(x)\sigma_3(y),
\\
\sigma_1L_x &= L_{\sigma_2(x)}\sigma_3,  
\\
\sigma_1R_y &= R_{\sigma_3(y)}\sigma_2,
\end{align*}
for any elements $x, y$ in $\A$. Recall that $L_x(y):= xy = R_y(x)$ for any $x,y\in \A$, are the left and right multiplication operators. 
The set $\teraut(\A)$ consisting of all ternary automorphisms of $\A$ is a group under the componentwise composition.



In what follows, we take a functorial approach to study ternary autormorphisms. This turns out to be very useful, since it has allowed us to derive some results about ternary derivations as consequences of our results about ternary automorphisms.



The subfunctor $\affteraut(\A)$ of the group functor $\affGL(\A)\times\affGL(\A)\times \affGL(\A)$ is defined by 
$\affteraut(\A)(S) = \teraut_S(\A_S)$, where $\A_S = \A\otimes_R S$ is the scalars extension of $\A$. After the suitable identifications,
$\affteraut(\A)(R)= \teraut(\A)$. 

Consider the extension by scalars $\A_{R[\varepsilon]} = \A \oplus \varepsilon \A$ of $\A$.
It is well known that a linear map $d$ on $\A$ is a derivation of $\A$ if and only if $\mathrm{Id} + \varepsilon d$ is an $R[\varepsilon]$-automorphism of $\A_{R[\varepsilon]}$. Similarly, in the ternary setting, we have:
$$(d_1, d_2, d_3) \in\terder(\A)\Leftrightarrow (\mathrm{Id} + \varepsilon d_1, \mathrm{Id} + \varepsilon d_2, \mathrm{Id} + \varepsilon d_3) \in \teraut_{R[\varepsilon]}(\A_{R[\varepsilon]}),$$
which means $$\hbox{\bf Lie}(\affteraut(\A))=\terder(\A),$$
where $\terder(\A)$ is the set consisting of all ternary derivations of $\A$. 
In other words, the Lie algebra of the group functor $\affteraut(\A)$ is $\terder(\A)$.

It will be also useful for our purposes the group functor $\affGL_1(\A)$ given by $\affGL_1(\A)(S) = (\A_S)^\times$, that is, the group of invertible elements in $\A_S$. 

Let $\op{\A}$ denote the opposite algebra of $\A$, and consider the group $(\A\times\op{\A})^\times$ with multiplication $(a, b)(a', b') = (aa', b'b)$. It is straightforward to check that the map $(\A\times\op{\A})^\times\to\teraut(\A)$ given by
$(x,y)\mapsto (R_yL_x,L_x,R_y)$ is a group monomorphism.

It was proved in \cite[2.5]{Petersson} that the set of all ternary automorphisms of an alternative algebra has a group structure, called the autotopy group of the algebras. Moreover, the autotopy group turned out to be isomorphic to to the structure group of the algebra \cite[2.7]{Petersson}.  

%

\begin{theorem} \label{Candidothm1}
Let $\A$ be an algebra. A triple $(\sigma_1, \sigma_2, \sigma_3)$, where $\sigma_i \in \GL(\A)$ for all $i$, 
is a ternary automorphism of $\A$ if and only if 
\[
(\sigma_1, \sigma_2, \sigma_3) = (R_yL_x \sigma, L_x \sigma, R_y \sigma),
\] 
for a unique automorphism $\sigma \in \aut(\A)$ and unique invertible elements $x, y \in \A^\times$.
There is a short exact sequence $$1\to(\A\times\op{\A})^\times\buildrel{i}\over{\to} \teraut(\A)\buildrel{p}\over{\to} \aut(\A)\to 1$$
where $i(x,y)=(R_yL_x,L_x,R_y)$ and $p(\sigma_1,\sigma_2,\sigma_3)=\sigma$.

\end{theorem}

\begin{proof}
Let $(\sigma_1,\sigma_2, \sigma_3)\in\teraut(\A)$. Write $z = \sigma_1(1)$, $x = \sigma_2(1)$ and
$y = \sigma_3(1)$. Then we have that $xy = z$, and so
\begin{align*}
\sigma_1(a) & = \sigma_2(a)y,
\\
\sigma_1(a) & = x \sigma_3(a),
\end{align*}
for all $a \in \A$. Thus $\sigma_1 = R_y \sigma_2 = L_x \sigma_3$, which implies that $L_x$ and $R_y$ are invertible endomorphisms.
We then have that $\sigma_2 = R_y^{-1}\sigma_1$ and $\sigma_3 = L_x^{-1}\sigma_1$, and
thus $\sigma_1(aa') = \sigma_1(a)y^{-1}x^{-1}\sigma_1(a')$, or equivalently,
$$x^{-1}\sigma_1(aa')y^{-1} = x^{-1}\sigma_1(a)y^{-1}x^{-1}\sigma_1(a')y^{-1},$$
which proves that $\sigma := L_{x^{-1}} R_{y^{-1}}\sigma_1\in\aut(\A)$. Showing uniqueness is straightforward and left it to the reader. 

Conversely, if $\sigma \in\aut(\A)$ and $x, y \in \A^\times$, then
$(R_yL_x \sigma, L_x \sigma, R_y \sigma) \in \teraut(\A)$. Lastly, the exactness of the short exact sequence clearly holds. 
\end{proof}

The following corollary was first proved in \cite[Lemma 1]{Sh1}. We obtain it here as an application of the Lie functor. 

\begin{corollary} \label{Candidothm2}
Any ternary derivation of an algebra $\A$ is of the form $$(d + L_x + R_y, d + L_x, d + R_y),$$ for a unique derivation $d$ of $\A$ and unique elements $x, y \in \A$.
\end{corollary}

\begin{proof}
Theorem \ref{Candidothm1} tells us that the map
$\varphi: \aut(\A)\times \A^\times\times \A^\times\to \teraut(\A)$, given by $\varphi(\sigma, x, y) =(R_yL_x \sigma, L_x \sigma, R_y \sigma)$ defines an isomorphism of $R$-modules. On the other hand, one can easily check that
$\aut(\A)\times \A^\times\times \A^\times$ becomes a group under the operation:
\[
(\sigma, x, y)(\tau, z, t) = \big(\sigma \tau, x \sigma(z), \sigma(t) y\big).
\]
Notice that the identity element is the triple $(\mathrm{Id}_\A, 1, 1)$ and that
\[
(\sigma, x, y)^{-1} = \big(\sigma^{-1}, \sigma^{-1}(x^{-1}), \sigma^{-1}(y^{-1})\big).
\]
Therefore, $\varphi$ is indeed a group isomorphism. This allows us to construct a group
functor isomorphism
$\affaut(\A)\times\affGL\nolimits_1(\A)\times\affGL\nolimits_1(\A)\cong\affteraut(\A),$
(here $\affaut(\A)$ refers to the group functor $S \mapsto \aut_S(\A_S)$).
Applying the Lie functor we obtain an $R$-algebra isomorphism:
\[
\psi: \der(\A)\times \A\times \A\to \terder(\A),
\]
defined by $\psi(d, x, y) = (d + L_x + R_y, d + L_x, d + R_y)$, for any $d \in \der(\A)$ and $x, y \in \A$. This finishes the proof.
\end{proof}

Automorphisms of triangular algebras were studied by Cheung \cite[Chapter 5]{Ch1}. For an algebra $\A$, we write $\inn(a)$ to denote the inner automorphism associated to the element $a \in \A$.

\begin{definition} \label{innerautdef}
A ternary automorphism of an algebra $\A$ is said to be \textbf{inner} if it is
of the form $\big(R_yL_x\inn(a), \, L_x\inn(a), \, R_y\inn(a)\big)$, where $x, y , a \in \A^\times$. The group of all inner ternary automorphisms of $\A$ will be denoted by $\innteraut(\A)$.

Notice that if the components of the inner ternary automorphism are equal, then $R_yL_x = L_x = R_y$, which implies that $L_x = R_y = \mathrm{Id}_\A$. So, the inner ternary automorphism becomes $(\inn(a),\inn(a),\inn(a)),$ as expected.

\end{definition}

%
%

\begin{definition} \cite[Definition 5.1.6]{Ch1}
Let $\T = \Tri(A, M, B)$ be a triangular algebra. An automorphism $\alpha$ of $\T$ is called
{\bf partible} if it can be written as $\alpha = \phi_z \bar \alpha$, where $\phi_z$ is the inner automorphism associated to $z \in \T^\times$, and $\bar \alpha$ is an automorphism of $\T$ satisfying that $\bar \alpha(A) = A$, $\bar \alpha(M) = M$ and $\bar \alpha(B) = B$.
\end{definition} 

It is straightforward to check that every automorphism $\sigma$ of a triangular algebra $\T = \Tri(A, M, B)$ satisfying
\begin{equation} \label{aut1}
\sigma(A) = A, \quad \sigma(B) = B \, \, \mbox{ and } \, \, \sigma(M) = M
\end{equation}
can be written as
\begin{equation} \label{aut2}
\sigma
\left(
\begin{array}{@{}cc@{}}
a & m  \\
  & b
\end{array}
\right) = \left(
\begin{array}{@{}cc@{}}
f_\sigma(a) & \nu_\sigma(m)  \\
  & g_\sigma(b)
\end{array}
\right),
\end{equation}
where $f_\sigma$, $g_\sigma$ are automorphisms of $A$ and $B$, respectively, and $\nu_\sigma: M \to M$ is a linear bijective map such that $\nu_\sigma(am) = f_\sigma(a)\nu_\sigma(m)$,
$\nu_\sigma(mb) = \nu_\sigma(m)g_\sigma(b)$, for all $a \in A$, $b \in B$, $m \in M$. This class of automorphisms has turned out to be very useful when studying $\sigma$-maps of triangular algebras; see, for instance, \cite{Bk3, MGRSO}.



We next determine the necessary and sufficient conditions to be imposed on a ternary automorphism $(\sigma_1, \sigma_2, \sigma_3)$ of $\T$ such that the associated automorphism $\sigma$ (see Theorem \ref{Candidothm1}), satisfies \eqref{aut1}, and to describe such ternary automorphisms. 

\begin{theorem}
Let $\T = \Tri(A, M, B)$ be a triangular algebra and $(\sigma_1, \sigma_2, \sigma_3) = (R_yL_x \sigma, L_x \sigma, R_y \sigma)$ a ternary automorphism of $\T$, where $x, y \in \T^\times$. Then $\sigma$ satisfies \eqref{aut1} if and only if $\sigma_2(A) = A$, $\sigma_1(M) = M$ and $\sigma_3(B) = B$. In such a case, writing $\sigma$ as in \eqref{aut2}, and
assuming
$x = \left(
\begin{array}{@{}cc@{}}
a_x & m_x  \\
  & b_x
\end{array}
\right)$, $y = \left(
\begin{array}{@{}cc@{}}
a_y & m_y  \\
  & b_y
\end{array}
\right)$ (where $a_x, a_y \in A^\times$ and $b_x, b_y \in B^\times$) we have that
\begin{align*}
\sigma_1
\left(
\begin{array}{@{}cc@{}}
a & m  \\
  & b
\end{array}
\right) & =
\left(
\begin{array}{@{}cc@{}}
R_{a_y}L_{a_x}f_\sigma(a) & a_x f_\sigma(a) m_y + a_x\nu_\sigma(m)b_y + m_x g_\sigma(b)b_y \\
  & R_{b_y}L_{b_x}g_\sigma(b)
\end{array}
\right),
\\
\sigma_2
\left(
\begin{array}{@{}cc@{}}
a & m  \\
  & b
\end{array}
\right) & =
\left(
\begin{array}{@{}cc@{}}
L_{a_x}f_\sigma(a) & a_x\nu_\sigma(m) + m_x g_\sigma(b) \\
  & L_{b_x}g_\sigma(b)
\end{array}
\right),
\\
\sigma_3
\left(
\begin{array}{@{}cc@{}}
a & m  \\
  & b
\end{array}
\right) & =
\left(
\begin{array}{@{}cc@{}}
R_{a_y}f_\sigma(a) & f_\sigma(a)m_y + \nu_\sigma(m)b_y  \\
  & R_{b_y}g_\sigma(b)
\end{array}
\right),
\end{align*}
for all $\left(
\begin{array}{@{}cc@{}}
a & m  \\
  & b
\end{array}
\right) \in \T$. Moreover, $(R_{a_y}L_{a_x}f_\sigma, L_{a_x}f_\sigma, R_{a_y}f_\sigma)$ is a ternary automorphism of $A$ and 
$(R_{b_y}L_{b_x}g_\sigma, L_{b_x}g_\sigma, R_{b_y}g_\sigma)$ is a ternary automorphism of $B$.
\end{theorem}

\begin{proof}
Let $x = \left(
\begin{array}{@{}cc@{}}
a_x & m_x  \\
 & b_x
\end{array}
\right)$ and $y = \left(
\begin{array}{@{}cc@{}}
a_y & m_y  \\
 & b_y
\end{array}
\right)$ where $a_x, a_y \in A^\times$, $m_x, m_y \in M$ and $b_x, b_y \in B^\times$. Notice that $a_x A = A$, $a_x M b_y = M$ and $B b_y = B$. From this, we obtain that $\sigma(A) = A$ if and only if
\begin{align*}
\sigma_2(A) =
\left(
\begin{array}{@{}cc@{}}
a_x & m_x  \\
  & b_x
\end{array}
\right)
\sigma(A) = \left(
\begin{array}{@{}cc@{}}
a_x A & 0  \\
  & 0
\end{array}
\right) = A.
\end{align*}
Similarly, we can see that $\sigma(M) = M$ if and only if $\sigma_1(M) = M$ and $\sigma(B) = B$ if and only if $\sigma_3(B) = B$. In fact, the condition that $\sigma_1(M) = M$ is equivalent to $\sigma_2(M) = M$, or, $\sigma_3(M) = M$.
\end{proof}

It was proven in \cite[Proposition 4.2]{MGRSO} that under certain mild hypothesis on a triangular algebra $\T$ (for instance, $A$ and $B$ are nondegenerate) every automorphism of $\T$ is $M$-preserving. As a consequence:

\begin{corollary}
Let $\T = \Tri(A, M, B)$ be a triangular algebra with $A$ and $B$ nondegenerate. Then: 
\begin{itemize}
\item[\rm (i)] Any ternary automorphism
$(\sigma_1, \sigma_2, \sigma_3)$ of $\T$ verifies $\sigma_1(M)= \sigma_2(M) = \sigma_3(M) = M$.
\item[\rm (ii)] Any ternary derivation $(d_1, d_2, d_3)$ of $\T$ satisfies $d_i(M)\subset M$ for $i=1,2,3$.
\end{itemize}
\end{corollary}

\begin{proof}
(i) By Theorem \ref{Candidothm1} there exist a unique automorphism $\sigma$ of $\T$ and unique invertible elements $x, y \in \T$ such that $(\sigma_1, \sigma_2, \sigma_3) = (R_yL_x\sigma, L_x\sigma, R_y\sigma)$. Now, \cite[Proposition 4.2]{MGRSO} applies to get that $\sigma(M) = M$. Then $\sigma_1(M) = xMy = M$, since $x$ and $y$ are invertible. Similarly, $\sigma_2(M) = \sigma_3(M) = M$.

\smallskip

(ii) Use now Corollary \ref{Candidothm2}.
\end{proof}

\section{Ternary derivations of triangular algebras}

In this section, we provide a precise description of ternary derivations of triangular algebras. We will also characterize when a given ternary derivation is inner. 

\begin{theorem} \label{terderchar}
Let $\T = \Tri(A, M, B)$ be a triangular algebra and $(d_1,d_2,d_3)$ be a triple of linear maps on $\T$. Then, $(d_1, d_2, d_3)$ is a ternary derivation of $\T$ if and only if it is of the following form:
\allowdisplaybreaks
\begin{align*}
d_1
\left(
\begin{array}{@{}cc@{}}
a & m  \\
  & b
\end{array}
\right) & =
\left(
\begin{array}{@{}cc@{}}
\delta_{1}(a) & an_1 {+} \tau_{1}(m) {+} n'_1 b
 \\ [1.2mm]
  & \mu_{1}(b)
\end{array}
\right),
\\ \label{td} \tag{\sc td}
d_2
\left(
\begin{array}{@{}cc@{}}
a & m  \\
  & b
\end{array}
\right) & =
\left(
\begin{array}{@{}cc@{}}
\delta_{2}(a) & an_2 {+} \tau_{2}(m) {+} n'_1 b
 \\ [1.2mm]
  & \mu_{2}(b)
\end{array}
\right),
\\
d_3
\left(
\begin{array}{@{}cc@{}}
a & m  \\
  & b
\end{array}
\right) & =
\left(
\begin{array}{@{}cc@{}}
\delta_{3}(a) & an_1 {+} \tau_{3}(m) {-} n_2 b
 \\ [1.2mm]
  & \mu_{3}(b)
\end{array}
\right),
\end{align*}
where $n_1, n'_1, n_2 \in M$ and $\delta_{i}: A \to A$, $\tau_{i}: M \to M$, $\mu_{i}: B \to B$, for $i = 1, 2, 3$,
are linear maps such that
\begin{itemize}
\item[{\rm (i)}] $(\delta_1, \delta_2, \delta_3)$ is a ternary derivation of $A$,
\item[{\rm (ii)}] $(\mu_1, \mu_2, \mu_3)$ is a ternary derivation of $B$,
\item[{\rm (iii)}] $\tau_1(am) = \delta_2(a)m + a \tau_3(m)$, for all $a \in A$ and $m \in M$,
\item[{\rm (iv)}] $\tau_1(mb) = m\mu_3(b) + \tau_2(m)b$, for all $b \in B$ and $m \in M$.
\end{itemize}
\end{theorem}

\begin{proof}
Let $(d_1, d_2, d_3)$ be a ternary derivation of $\T$. By Corollary \ref{Candidothm2} there exist a unique derivation $d$ of $\T$ and elements $x, y \in \T$ such that 
\[
(d_1, d_2, d_3) = (d + L_x + R_y, \, d + L_x, \, d + R_y).
\]
Write $x = \left(
\begin{array}{@{}cc@{}}
a_x & m_x  \\
  & b_x
\end{array}
\right)$ and $y = \left(
\begin{array}{@{}cc@{}}
a_y & m_y  \\
  & b_y
\end{array}
\right)$. 
An application of \cite[Proposition 2.2]{FM} tells us that 
\begin{equation} \label{derivation}
d
\left(
\begin{array}{@{}cc@{}}
a & m  \\
  & b
\end{array}
\right)  =
\left(
\begin{array}{@{}cc@{}}
\delta(a) & an {+} \tau(m) {-} n b
 \\ [1.2mm]
  & \mu(b)
\end{array}
\right),
\end{equation}
where $\delta$ and $\mu$ are derivations of $A$ and $B$, respectively, $n \in M$ and $\tau: M \to M$ is a linear map
satisfying that $\tau(am) = \delta(a)m + a \tau(m)$, $\tau(mb) = m \mu(b) + \tau(m)b$, for all $a\in A$, $m \in M$ and $b \in B$. We then have that 
\allowdisplaybreaks
\begin{align*}
d_1
\left(
\begin{array}{@{}cc@{}}
a & m  \\
  & b
\end{array}
\right)  & {=} \left(
\begin{array}{@{}cc@{}}
\Big(\delta {+} L_{a_x} {+} R_{a_y}\Big)(a) & a(n + m_y) {+} \Big(\tau + L_{a_x} {+} R_{b_y}\Big)(m) {+} (-n + m_x)b
 \\ [1.2mm]
  & \Big(\mu {+} L_{b_x} {+} R_{b_y}\Big)(b)
\end{array}
\right),
\\
d_2
\left(
\begin{array}{@{}cc@{}}
a & m  \\
  & b
\end{array}
\right)& {=} \left(
\begin{array}{@{}cc@{}}
\Big(\delta {+} L_{a_x}\Big)(a) & an {+} \Big(\tau + L_{a_x} \Big)(m) + (-n {+} m_x)b
 \\ [1.2mm]
  & \Big(\mu {+} L_{b_x}\Big)(b)
\end{array}
\right),
\\
d_3
\left(
\begin{array}{@{}cc@{}}
a & m  \\
  & b
\end{array}
\right)  & {=} \left(
\begin{array}{@{}cc@{}}
\Big(\delta + R_{a_y}\Big)(a) & a(n + m_y) {+} \Big(\tau + R_{b_y}\Big)(m) -nb
 \\ [1.2mm]
  & \Big(\mu + R_{b_y}\Big)(b)
\end{array}
\right).
\end{align*}
Corollary \ref{Candidothm2} yields that $(\delta_1, \delta_2, \delta_3) = \Big(\delta + L_{a_x} + R_{a_y}, \delta + L_{a_x}, \delta + R_{a_y}\Big)$ and 
$(\mu_1, \mu_2, \mu_3) = \Big(\mu + L_{b_x} + R_{b_y}, \mu + L_{b_x}, \mu + R_{b_y}\Big)$ are ternary derivations of $A$ and $B$, respectively. This proves (i) and (ii). 

Moreover, it is straightforward to check that the triple of linear maps
\[
(\tau_1, \tau_2, \tau_3) = \Big(\tau + L_{a_x} + R_{b_y}, \tau + L_{a_x}, \tau + R_{b_y}\Big)
\]
satisfies (iii) and (iv). To finish, let $n_1 = n + m_y$, $n_1' = -n + m_x$ and $n_2 = n$. 

The converse trivially holds. 
\end{proof}


\subsection{Inner ternary derivations}
Let $\A$ be an $R$-algebra and $d$ be a derivation of $\A$. Recall that a derivation of $\A$ is said to be {\bf inner} if $d = R_a - L_a$ for some $a \in \A$. At this point, a natural question arises: ``what should an inner ternary derivation be?'' The answer to this question will come from the Lie algebra of the group functor $\affinnteraut(\A)$ of inner ternary automorphisms, defined by 
\[
\affinnteraut(\A)(S):=\innteraut(\A_S),
\] 
for any $\alg_R$. Notice that $\affinnteraut(\A)(R) = \innteraut(\A)$. 

 Then, a triple $(d_1, d_2, d_3)$ of linear maps of $\A$ is an inner ternary derivation of $\A$ if and only if 
$(\mathrm{Id} + \varepsilon d_1, \mathrm{Id} + \varepsilon d_2, \mathrm{Id} + \varepsilon d_3)$ is an inner automorphism of $\A_{R[\varepsilon]}$. So, imposing the conditions of Definition \ref{innerautdef} to $(\mathrm{Id} + \varepsilon d_1, \mathrm{Id} + \varepsilon d_2, \mathrm{Id} + \varepsilon d_3)$, after a long easy but routinely calculation we get that 
$$d_1 = L_a + R_b,  \quad d_2 = L_a + R_c, \quad d_3 = -L_c + R_b,$$
for some $a, b, c \in \A$.
This motivates the following definition: 

\begin{definition} \cite{Sh1} \label{innerter}
A ternary derivation $(d_1, d_2, d_3)$ of an algebra $\A$ is called an {\bf inner ternary derivation} if there exist $a, b, c \in \A$ such that
\[
(d_1, d_2, d_3) = (L_a + R_b, \, L_a + R_c,\, -L_c + R_b).
\]
\end{definition}
Notice that $d_1 = d_2 = d_3$ implies $a = -b = -c$ so $d_1 = R_b - L_b$. This shows that the derivation $d_1$ is inner, as expected. 

We write $\mathrm{InnTerDer}(\A)$ for the set consisting of all inner ternary derivations of $\A$. A characterization of inner ternary derivations follows. 

\begin{theorem} \label{innerterder}
A ternary derivation $(d + L_x + R_y, d + L_x, d + R_y)$ of an algebra $\A$ is inner if and only if $d$ is an inner derivation of $\A.$
\end{theorem}

\begin{proof}
First suppose that the ternary derivation $(d + L_x + R_y, d + L_x, d + R_y)$ is inner. Then there exist $a, b, c \in \A$ such that
\begin{align*}
    d + L_x + R_y &= L_a + R_b, \\
    d + L_x &= L_a + R_c, \\
    d + R_y &= -L_c + R_b.
\end{align*}
Subtracting the latter two equations from the first, we obtain
\begin{align*}
R_y &= L_a + R_b - L_a - R_c = R_{b-c}, \\
L_x &= L_a + R_b + L_c - R_b = L_{a+c}.
\end{align*}
Substituting these values for $L_x$ and $R_y$ into the first equation, we see that $d = R_c - L_c.$ Hence, $d$ is an inner derivation of $\A.$

Now assume that $d$ is an inner derivation of $\A.$ So there is some $c \in \A$ such that $d = R_c - L_c.$ Then,
\begin{align*}
    d + L_x + R_y &= L_{x-c} + R_{c+y}, \\
    d + L_x &=  L_{x-c} + R_c, \\
    d + R_y &= - L_c + R_{c+y}.
\end{align*}
Taking $a = x -c$ and $b = c + y,$ we can see that $(d + L_x + R_y, d + L_x, d + R_y) = (L_a + R_b, L_a + R_c, -L_c + R_b)$ for some $a, b, c \in \A.$ Hence the ternary derivation is inner.
\end{proof}

Now we move to the case of inner ternary derivations of triangular algebras.

\begin{theorem}
Let $\T = \Tri(A,M,B)$ be a triangular algebra with $M$ faithful and $(d_1,d_2,d_3)$ be a ternary derivation of $\T$ of the form given by \eqref{td}. Then $(d_1,d_2,d_3) \in \mathrm{InnTerDer}(\T)$ if and only if there is some $i \in \{1,2,3\}$ such that $\tau_i (m) = a_1 m + m b_1$ for some fixed $a_1 \in A$ and $b_1 \in B$.
\end{theorem}

\begin{proof}
By Corollary \ref{Candidothm2} there exist a unique derivation $d$ of $\T$ and unique elements $x, y \in \T$ such that $(d_1, d_2, d_3) = (d + L_x + R_y, d + L_x, d + R_y)$. Theorem \ref{innerterder} yields $(d_1, d_2, d_3)\in \mathrm{InnTerDer}(\T)$ if and only if $d$ is an inner derivation of $\T$. 
Now, suppose that $d$ is written as in \eqref{derivation}. Then \cite[Proposition 2.2.3]{Ch1} implies that $d$ is inner if and only if $\tau(m) = a_0 m + m b_0$ for some $a_0 \in A$ and $b_0 \in B$.

First suppose that $\tau(m) = a_0 m + m b_0$ for some fixed $a_0 \in A$ and $b_0 \in B.$ From the proof of Theorem \ref{terderchar}, we obtain
\begin{equation*}
\tau_1(m) = \tau(m) + a_x m + m b_y = (a_0 + a_x) m + m (b_y + b_0).
\end{equation*}
Taking $a_1 = a_0 + a_x$ and $b_1 = b_y + b_0$ we have $\tau_i (m) = a_1 m + m b_1$ for $i = 1$.

Now assume that $\tau_i(m) = a_1 m + m b_1$ for some $i \in \{1,2,3\}$. If it holds for $i = 1,$ then from the proof of Theorem \ref{terderchar}, we obtain
\begin{equation*}
\tau(m) = \tau_1(m) - a_x m - m b_y = (a_1 -a_x) m + m (b_1 - b_y).
\end{equation*}
Taking $a_0 = a_1 - a_x$ and $b_0 = b_1 - b_y,$ we obtain the desired result. Using a similar method we can prove the result for $i = 2$ and $i = 3$.
\end{proof}


\subsection{Components of a ternary derivation}

It was proved in \cite[Corollary 2.2.2]{Ch1} that the faithfulness of the bimodule yields a weaker characterization of derivations of $\T$. More precisely, a linear map $d$ of a triangular algebra $\T = \Tri(A, M, B)$ with $M$ faithful, is a derivation of $\T$ if and only if it can be written as in \eqref{derivation} and satisfies that $\tau(am) = \delta(a)m + a \tau(m)$, $\tau(mb) = m \mu(b) + \tau(m)b$, for all $a\in A$, $m \in M$ and $b \in B$. In other words, $\delta$ and $\mu$ do not need to be derivations of $A$ and $B$, respectively. This is no longer the case for ternary derivations:

\begin{example}
Let $\A$ be an algebra and $d$ a (nonzero) derivation of $\A$. Consider the triangular algebra $\T = \Tri(\A, \A, \A)$. Clearly, $\A$ is a faithful $(\A, \A)$-bimodule.
The triple of linear maps $(d_1, d_2, d_3)$, where
\allowdisplaybreaks
\begin{align*}
d_1
\left(
\begin{array}{@{}cc@{}}
a & m  \\
  & b
\end{array}
\right) & = 
\left(
\begin{array}{@{}cc@{}}
0 &  d(m) 
 \\ [1.2mm]
  & 0
\end{array}
\right), 
\\ 
d_2
\left(
\begin{array}{@{}cc@{}}
a & m  \\
  & b
\end{array}
\right) & =
\left(
\begin{array}{@{}cc@{}}
d(a) & d(m) 
 \\ [1.2mm]
  & 0
\end{array}
\right),
\\
d_3
\left(
\begin{array}{@{}cc@{}}
a & m  \\
  & b
\end{array}
\right) & =
\left(
\begin{array}{@{}cc@{}}
0 & d(m) 
 \\ [1.2mm]
  & d(b)
\end{array}
\right),
\end{align*}
satisfies Conditions (iii) and (iv) of Theorem \ref{terderchar} and neither $(0, d, 0)$ nor $(0, 0, d)$ are ternary derivations of $\A$.
\end{example}

The previous example inspired us to search for a relationship among the components of a ternary derivation of an arbitrary algebra $\A$. To be more precise, if two ternary derivations of $\A$ share the second (respectively, third) component, is there any relationship between their first and third (respectively, first and second) components?

\smallskip

Let $\A$ be an algebra. For $a, b \in \A$, we write $T_{a, b} = L_a - R_b$. 

\begin{proposition}\label{melon}
For $a, b, c, d, e, f \in \A$, the following assertions hold: 
\begin{enumerate}
\item[\rm (i)] If $a, b, c, d, e, f \in \A^\times$, then $(L_aR_b, L_cR_d, L_eR_f) \in \teraut(\A)$ if and only if $c^{-1}a, bf^{-1}, de \in Z(\A)$ and $ab = cdef$.

\smallskip

\item[\rm (ii)] $(T_{a,b}, T_{c,d}, T_{e,f}) \in \terder(\A)$ if and only if $a-c, b-f, e-d\in Z(\A)$ and $a + f + d = c + b + e$.
\end{enumerate}
\end{proposition}

\begin{proof}
For (i), let $a, b, c, d, e, f \in \A^\times$. Assume first that $(L_aR_b, L_cR_d, L_eR_f) \in \teraut(\A)$. Then 
\begin{equation} \label{watermelon}
a(xy)b = (cxd)(eyf),   
\end{equation}
for all $x, y \in \A$.
Taking $x = y = 1$ in \eqref{watermelon} gives $ab = cdef$ and so $bf^{-1} = a^{-1}cde$. Multiplying \eqref{watermelon} by $a^{-1}$ on the left, by $f^{-1}$ on the right and setting $x = 1$ we get
$ybf^{-1} = a^{-1}c dey = bf^{-1}y$, which proves that $bf^{-1} \in Z(\A)$. Similarly, one can prove that $de \in Z(\A)$; lastly, $c^{-1}a \in \Z(A)$ follows from $ab = cdef$.

Conversely, suppose that $c^{-1}a, bf^{-1}, de \in Z(\A)$ and $ab = cdef$. For $x, y \in \A$ we have 
\begin{align*}
(L_aR_b)(xy) & = a(xy)b = (cdefb^{-1})xy b = (cde)(xy)(fb^{-1}b) = c(de)x(yf) =
\\
& = (cxd)(eyf) = (L_cR_d)(x)(L_eR_f)(y),
\end{align*}
finishing the proof.

\smallskip 

Then (ii) follows from (i) by an application of the Lie functor.
\end{proof}

\begin{corollary} 
\, \, \hfill
\begin{enumerate}
\item[\rm (i)]
The triples $(L_a R_b, L_a R_c, L_{c^{-1}}R_b)$, $(L_a, L_a, \mathrm{Id}_\A)$, $(R_b, \mathrm{Id}_\A, R_b)$ are ternary automorphisms of $\A$, for all $a, b, c \in \A^\times$.
\smallskip
\item[\rm (ii)]
The triples $(T_{a,b},T_{a,c},T_{c,b})$, $(L_a, L_a, 0)$, $(R_b, 0, R_b)$ are ternary derivations of $\A$, for all $a, b, c \in \A$. 
\end{enumerate}
\end{corollary}

Since we want to investigate ternary derivations sharing the second (respectively, third) components, it is natural to look at the ternary derivations whose second (respectively, third) components are 0. A few consequences of our previous results follow.

\begin{corollary} \label{0secondcomponent}
Let $\A$ be an algebra. 
\begin{itemize}
\item[\rm(i)] The only ternary automorphism of $\A$ whose second component is $1_\A$ are of the form $(R_a, 1_\A, R_a)$ for some $a \in \A^\times$.
\smallskip
\item[\rm(ii)] The only ternary automorphisms of $\A$ with third second component is $1_\A$ are of the form $(L_a, L_a, 1_\A)$ for some $a \in \A^\times$.
\item[\rm(iii)] The only ternary derivations of $\A$ with zero second component are of the form $(R_a, 0, R_a)$ for some $a \in \A$.
\smallskip
\item[\rm(iv)] The only ternary derivations of $\A$ with zero third component are of the form $(L_a, L_a, 0)$ for some $a \in \A$.
\end{itemize}
\end{corollary}  

\begin{proof}
The proofs of (i) and (ii) are trivial, and (iii) and (iv) follow
from an application of the Lie functor. 
\end{proof}

\begin{corollary} \label{component2}
Let $\A$ be an algebra and $(d_1, d_2, d_3), (d_1', d_2, d_3') \in \terder(\A)$. Then $d_1 - d_1' = d_3 - d_3' = R_b$ for some $b\in \A$.
Conversely, if $(d_1,d_2,d_3) \in \terder(\A)$ satifies $d_1 - d_1' = d_3 - d_3' = R_b$,  (for some $b \in \A$) then $(d_1',d_2,d_3') \in \terder(\A)$. 
\end{corollary} 

\begin{corollary}
Let $\A$ be an algebra and $(d_1, d_2, d_3) \in \terder(\A)$. Then 
\[
\{(d_1 + R_a, d_2, d_3 + R_a)| \, \, a\in \A\}
\] 
is the set consisting of all ternary derivations of $\A$ whose second component is $d_2$.
\end{corollary}

\begin{corollary} \label{component3}
Let $\A$ be an algebra and $(d_1, d_2, d_3), (d_1', d'_2, d_3)  \in \terder(\A)$. Then $d_1 - d_1' = d_2 - d_2' = L_a$ for some $a\in \A$.
Conversely, if $(d_1,d_2,d_3)  \in \terder(\A)$ satisfies $d_1 - d_1' = d_2 - d_2' = L_a$ (for some $a \in \A$) then $(d_1',d_2',d_3) \in \terder(\A)$.
\end{corollary}

\begin{lemma} \label{julve}
Let $\A$ be an algebra. 
\begin{enumerate}
\item[\rm (i)]
If $(\sigma_1, \sigma_2, \sigma_3) \in \teraut(\A)$, then $\sigma_j(xy)=\sigma_j(x)\sigma_j(1)^{-1}\sigma_j(y)$, for $j = 1, 2, 3$ and any $x, y \in \A$.
\smallskip
\item[\rm (ii)] 
If $(d_1, d_2, d_3) \in \terder(\A)$,  then $d_j(xy) - d_j(x)y - xd_j(y) = -x d_j(1)y$, for $j = 1, 2, 3$ and any $x, y\in\A$. 
\end{enumerate}
\end{lemma}

\begin{proof} 
(i) We prove the result for $j = 1$. The other cases are similar. 

\noindent Notice that $\sigma_1 = R_b \sigma_2 = L_a \sigma_3$ for the invertible elements $a = \sigma_2(1)$ and $b = \sigma_3(1)$ of $\A$. From here we derive that 
$\sigma_1(xy) = \sigma_1(x)(ab)^{-1}\sigma_1(y) = \sigma_1(x)\sigma_1(1)^{-1}\sigma_1(y)$, as desired.
\smallskip

(ii) follows from (i) by an application of the Lie functor. See \cite[Lemma 1]{Sh1} for a direct proof for the case of $j = 1$.

\end{proof}

\begin{remark}\rm 
There is a sort of  converse of Lemma \ref{julve} (i): if $\sigma_1\in\GL(\A)$ satisfies
$\sigma_1(xy) = \sigma_1(x) \sigma_1(1)^{-1} \sigma_1(y)$ for any $x, y \in \A$, and $\sigma (1) = ab$ for some
$a, b \in\A^\times$, then $(\sigma, R_{b^{-1}}\sigma,  L_{a^{-1}}\sigma) \in \teraut(\A)$.
\end{remark}

\begin{lemma} \label{julve2}
Let $\A$ be an algebra and $(d_1, d_2, d_3)$ a triple of linear maps of $\A$. 
\begin{itemize}
\item[\rm (i)] If $d_1$ satisfies Lemma \ref{julve} {\rm (i)}, then $\big(d_1, d_1 - R_a, d_1 - L_{d_1(1) - a} \big) \in \terder(\A)$ for all $a \in A$.
\smallskip
\item[\rm (ii)] If $d_2$ satisfies Lemma \ref{julve} {\rm (i)}, then $\big(d_2 {+} R_a, d_2, d_2 + R_a {-} L_{d_2(1)}\big) {\in} \terder(\A)$ for all $a \in A$.
\smallskip
\item[\rm (iii)] If $d_3$ satisfies Lemma \ref{julve} {\rm (i)}, then $\big(d_3 {+} L_a, d_3 {+} L_a {-} R_{d_3(1)}, d_3\big) {\in} \terder(\A)$ for all $a \in A$.
\end{itemize}
\end{lemma}

\begin{proof}
We prove (ii): let $(d_1, d_2, d_3)$ be a triple of linear maps of $\A$ and $a \in \A$. For $x, y \in \A$ we have that 
\begin{align*}
d_2(x)y + x\big(d_2 + R_a - L_{d_2(1)}\big)(y) &= d_2(x)y + x d_2(y) + xya - xd_2(1)y
\\
& \stackrel{\rm Lemma \, \ref{julve} \, (i)}{=} d_2(xy) + xya = (d_2 + R_a)(xy),
\end{align*}
which shows that $\big(d_2 + R_a, \, d_2, \, d_2 + R_a - L_{d_2(1)}\big) \in \terder(\A)$, as desired. 
\end{proof}

We are now in a position to prove the analogue weaker caracterization of derivations of triangular algebras with faithful bimodule for ternary derivations:

\begin{theorem} \label{components}
Let $\T = \Tri(A, M, B)$ be a triangular algebra with $M$ faithful, and $(d_1, d_2, d_3)$ a triple of linear maps of $\T$ written as in 
\eqref{td}. Suppose that Conditions (iii) and (iv) of Theorem \ref{terderchar} are satisfied. Then there exist linear maps $\delta'_1, \delta'_3: A \to A$ and $\mu'_1, \mu'_2: B \to B$ such that $(\delta'_1, \delta_2, \delta'_3)$ and $(\mu'_1, \mu'_2, \mu_3)$ are ternary derivations of $A$ and $B$, respectively. Moreover, the triple $(d'_1, d'_2, d'_3)$ where 
\allowdisplaybreaks
\begin{align*}
d'_1
\left(
\begin{array}{@{}cc@{}}
a & m  \\
  & b
\end{array}
\right) & =
\left(
\begin{array}{@{}cc@{}}
\delta'_{1}(a) & an_1 {+} \tau_{1}(m) {+} n'_1 b
 \\ [1.2mm]
  & \mu'_{1}(b)
\end{array}
\right),
\\
d'_2
\left(
\begin{array}{@{}cc@{}}
a & m  \\
  & b
\end{array}
\right) & =
\left(
\begin{array}{@{}cc@{}}
\delta_{2}(a) & an_2 {+} \tau_{2}(m) {+} n'_1 b
 \\ [1.2mm]
  & \mu'_{2}(b)
\end{array}
\right),
\\
d'_3
\left(
\begin{array}{@{}cc@{}}
a & m  \\
  & b
\end{array}
\right) & =
\left(
\begin{array}{@{}cc@{}}
\delta'_{3}(a) & an_1 {+} \tau_{3}(m) {-} n_2 b
 \\ [1.2mm]
  & \mu_{3}(b)
\end{array}
\right),
\end{align*}
is a ternary derivation of $\T$. 
\end{theorem}

\begin{proof} 
Notice that Conditions (iii) and (iv) of Theorem \ref{terderchar} can be rewritten as $\tau_1 = \tau_3 + L_{\delta_2(1_A)}$ and $\tau_1 = \tau_2 + R_{\mu_3(1_B)}$, respectively. Given $a, a' \in A$ and $m \in M$, Condition (iii) of Theorem \ref{terderchar} applies to get  
\begin{align}
\tau_1(aa'm) & = \delta_2(aa')m + aa'\tau_3(m), \label{one}
\\
\tau_1(aa'm) & = \delta_2(a)a'm + a\tau_3(a'm), \label{two}
\\
a\tau_1(a'm) & = a\delta_2(a')m + aa'\tau_3(m). \label{three}
\end{align}
Now, $\eqref{one}-\eqref{two}-\eqref{three}$ gives 
\[
-a\tau_1(a'm) = \big(\delta_2(aa') - \delta_2(a)a' - a\delta_2(a')\big)m - a\tau_3(a'm),
\]
which implies that 
\[ 
a\tau_3(a'm) - a\tau_1(a'm) = \big(\delta_2(aa') - \delta_2(a)a' - a\delta_2(a')\big)m.
\]
Now, from $\tau_3 = \tau_1 - L_{\delta_2(1_A)}$ we derive that 
\[
-a \delta_2(1_A) a'm = \big(\delta_2(aa')- \delta_2(a)a'- a\delta_2(a')\big).
\]
So $- a\delta_2(1_A)a' = \delta_2(aa') - \delta_2(a)a' - a\delta_2(a')$, since $M$ is faithful. Lemma \ref{julve2} (ii) yields that 
the triple $(\delta_2 + R_a, \, \delta_2, \, \delta_2 + R_a - L_{\delta_2(1_A)})$ is a ternary derivation of $A$ for all $a \in A$. 
Similarly, using now Lemma \ref{julve2} (iii), one can prove that $(\mu_3 + L_b, \, \mu_3 + L_b - R_{\mu_3(1_B)}, \, \mu_3)$ is a ternary derivation of $B$ for all $b \in B$. An application of Theorem \ref{terderchar} concludes the proof.
\end{proof}

\section{Idempotent preserving automorphisms} 

In this section we investigate the automorphisms of a triangular algebra $\T$ preserving the idempotent $p$ (and therefore the idempotent $q$). Notice that, in general, an automorphism of $\T$ does not need to preserve $p$. For example, consider $\T = \Tri(T_2(R), T_2(R), T_2(R))$, where $T_2(R)$ denotes the square upper triangular matrices with coefficients in $R$. J\o ndrup proved in \cite[page 210]{Jn} that the map $\sigma: \T \to \T$ defined by 
\[
\sigma
\left(
\begin{array}{@{}cc@{}}
\left(
\begin{array}{@{}cc@{}}
a & c  \\
0  & b
\end{array}
\right) & \left(
\begin{array}{@{}cc@{}}
d & f  \\
0  & e
\end{array}
\right)  \\
&
\\
  & \left(
\begin{array}{@{}cc@{}}
g & i  \\
0 & h
\end{array}
\right)
\end{array}
\right) = 
\left(
\begin{array}{@{}cc@{}}
\left(
\begin{array}{@{}cc@{}}
a & d  \\
0  & g
\end{array}
\right) & \left(
\begin{array}{@{}cc@{}}
c & f  \\
0  & i
\end{array}
\right)  \\
&
\\
  & \left(
\begin{array}{@{}cc@{}}
b & e  \\
0  & h
\end{array}
\right)
\end{array}
\right),
\]
is an automorphism of $\T$. We have that
\begin{align*}
\sigma(p) = 
\sigma
\left(
\begin{array}{@{}cc@{}}
\left(
\begin{array}{@{}cc@{}}
1 & 0  \\
0  & 1
\end{array}
\right) & \left(
\begin{array}{@{}cc@{}}
0 & 0  \\
0  & 0
\end{array}
\right)  \\
&
\\
  & \left(
\begin{array}{@{}cc@{}}
0 & 0  \\
0 & 0
\end{array}
\right)
\end{array}
\right) = 
\left(
\begin{array}{@{}cc@{}}
\left(
\begin{array}{@{}cc@{}}
1 & 0  \\
0  & 0
\end{array}
\right) & \left(
\begin{array}{@{}cc@{}}
0 & 0  \\
0  & 0
\end{array}
\right)  \\
&
\\
  & \left(
\begin{array}{@{}cc@{}}
1 & 0  \\
0 & 0
\end{array}
\right)
\end{array}
\right),
\end{align*}
that is, $\sigma$ does not preserve $p$. 

\begin{proposition}\label{thames}
Let $\T = \Tri(A, M, B)$ be a triangular algebra. 
\begin{enumerate}
\item[\rm (i)] Assume that $A$ and $B$ are nondegenerate and that $M$ is faithful. Then any automorphism of $\T$ can be written as the composition of an inner automorphism with an automorphism preserving the idempotent $p$.
\smallskip
\item[\rm (ii)] Any derivation of $\T$ can be written as the sum of an inner derivation plus a derivation vanishing on the idempotent $p$ (and therefore on $q$).
\end{enumerate}
\end{proposition}

\begin{proof}
(i) Let $\sigma$ be an automorphism of $\T$. \cite[Proposition 4.2]{MGRSO} applies to get that $\sigma(M) = M$. Then from $pm = m$ and $mp = 0$ for all $m\in M$, we derive that
\begin{align} 
\sigma(p)m & = m, \quad \forall \, m \in M, \label{sigmap1} \\ 
m \sigma(p) & = 0,\quad \forall \, m \in M. \label{sigmap2}
\end{align} 
Suppose that  
$\sigma(p) = \left(
\begin{array}{@{}cc@{}}
a_p & m_p  \\
  & b_p
\end{array}
\right) \in \T$. Applications of \eqref{sigmap1} and \eqref{sigmap2} yield that $\sigma(p) = \left(
\begin{array}{@{}cc@{}}
1_A & m_p  \\
  & 0
\end{array}
\right)$, since $M$ is faithful. 

Let $t: = \left(
\begin{array}{@{}cc@{}}
1_A & -m_p  \\
  & 1_B
\end{array}
\right) \in \T^\times$. Then $\sigma = \inn(t)\tau$, where $\tau = \inn(t^{-1}) \sigma$ preserves $p$. This finishes the proof of (i). 

\smallskip

(ii) Let $d$ be a derivation of $\T$. An application of \cite[Proposition 2.2]{FM} gives that 
$d(p)= \left(
\begin{array}{@{}cc@{}}
0 & n  \\
  & 0
\end{array}
\right)$, for some $n \in M$. Let $d_1: = d + \ad(d(p))$, then $d = d_1 + \ad(-d(p))$ and $d_1(p) = d(p) + [d(p), p] = 2d(p)p = 0$, 
as desired.
\end{proof}


\begin{corollary}
Let $\T = \Tri(A, M, B)$ be a triangular algebra. Then any derivation $d$ of $\T$ satisfies $d(M)\subseteq M$.
\end{corollary}

\begin{proof}
Let $d$ be a derivation of $\T$. If $d$ is inner, then the result clearly follows. If $d(p) = d(q) = 0$, then $d(M) \subseteq M$ since $M = p\T q$. Proposition \ref{thames} (ii) concludes the proof.
\end{proof}

\begin{notation}
{\rm 
For an algebra $\A$, we write $\outder(\A)$ to denote the quotient (Lie) algebra $\der(\A)/\innder(\A)$, where $\innder(\A)$ is the (Lie) ideal 
of $\der(\A)$ of all the inner derivations of $\A$. The elements of $\outder(\A)$ are called {\bf outer derivations}.}
\end{notation}

\begin{corollary}
Let $\T$ be a triangular algebra. Then 
\[
\outder(\T) \cong \der_0(\T)/ \innder_0(\T),
\] 
where $\der_0(\T)$ is the Lie algebra consisting on all the derivations vanishing at $p$ and $\innder_0(\T) = \innder(\T) \cap \der_0(\T)$.
\end{corollary}

\begin{corollary}
Let $\T = \Tri(A, M, B)$ be a triangular algebra.
\begin{itemize}
\item[\rm (i)] Assume that $A$ and $B$ are nondegenerate and that $M$ is faithful as a left $A$-module and as a right $B$-module. Then every ternary automorphism of $\T$ can be written as the composition of an inner ternary automorphism and a ternary automorphism of the form $(\sigma, \sigma, \sigma)$, where $\sigma$ is an automorphism of $\T$ preserving $p$.

\smallskip

\item[\rm (ii)] Any ternary derivation of $\T$ can be expressed as the sum of an inner ternary derivation plus a ternary derivation of the form $(d, d, d)$, where $d$ is a derivation of $\T$ such that $d(p) = 0$.
\end{itemize}
\end{corollary}

\begin{proof}
(i) follows from Theorem \ref{Candidothm1} and Proposition \ref{thames} (i). To prove (ii) 
apply Corollary \ref{Candidothm2} and Proposition \ref{thames} (ii).
\end{proof}

Let $\T = \Tri(A, M, B)$ be triangular algebra.  We write $\aut_0(\T)$ to denote the group of automorphisms of $\T$ preserving $p$. Using the Peirce decomposition of $\T$ associated to $p$ one can easily check that $\aut_0(\T)$ contains all the automorphisms $\sigma$ of $\T$ preserving $A$, $M$ and $B$, that is, $\sigma(A) = A$, $\sigma(M) = M$, $\sigma(B) = B$. The importance of these latter automorphisms together with Proposition \ref{thames} (i) have motivated us to describe $\aut_0(\T)$.

\begin{proposition}
Let $\T = \Tri(A, M, B)$ be triangular algebra. Then 
\[
\aut_0(\T)/ \mathcal{K} \cong \GL_{A, B}(M),
\]
where $\GL_{A, B}(M)$ is the following subgroup of $\GL(M)$
\begin{align*}
\GL_{A, B}(M) = \big \{& T \in \GL(M) | \, \,  \exists \, \alpha_T \in \aut(A), \beta_T \in \aut(B): 
\\
& TL_aR_b = L_{\alpha_T(a)}R_{\beta_T(b)}T, \, \, \forall \, a \in A, b \in B \big \},
\end{align*}
and 
$\mathcal{K} = \big \{ (\alpha, \beta)\in \aut(A) \times \aut(B) | \, \, (\Imm(\alpha) - 1_A)M = 0 = M(\Imm (\beta) - 1_B)\big \}.$
\end{proposition}

\begin{proof}
It is straightforward to check that $\GL_{A, B}(M)$ is indeed a subgroup of $\GL(M)$. Define $\Psi: \aut_0(\T) \to \GL_{A, B}(M)$ by $\Psi(\sigma) = \sigma|_M$, for all $\sigma \in \aut_0(\T)$. Note that $\Psi$ is well defined since every element of $\aut_0(\T)$ is $M$-preserving, and it is clearly a group homomorphism. We claim that its kernel is isomorphic to $\mathcal{K}$. In fact, let $\sigma \in \ker \Psi$. Then, $\sigma(M) = 0$ and so
$\sigma
\left(
\begin{array}{@{}cc@{}}
a & m  \\
0  & b
\end{array}
\right) = 
\left(
\begin{array}{@{}cc@{}}
\sigma_A(a) & 0  \\
0  & \sigma_B(b)
\end{array}
\right)$, for some $\sigma_A \in \aut(A)$ and $\sigma_B \in \aut(B)$. Trivially, $(\Imm(\sigma_A) - 1_A)M = 0 = M(\Imm (\sigma_B) - 1_B)$. The map $\ker \Psi \to \mathcal{K}$ given by $\sigma \mapsto (\sigma_A, \sigma_B)$ is the desired isomorphism. 
\end{proof}

\begin{corollary} \label{aut0}
Let $\T = \Tri(A, M, B)$ be triangular algebra with $M$ faithful. Then 
$\aut_0(\T) \cong \GL_{A, B}(M)$. 
\end{corollary}


\begin{theorem} \label{findingaut0}
Let $\T = \Tri(A, M, B)$ be triangular algebra with $M$ faithful.
Let $\mathcal{H }= \hom_R(A \otimes M \otimes B, M)$ be the set of all $R$-modules homomorphisms from $A \otimes M \otimes B$ to $M$. Define $f: \GL(M) \to \mathcal{H}$ by $f(T) = T\mu (\Id_A \otimes T^{-1} \otimes \Id_B)$ (for all $T \in \GL(M)$) and $g: \aut(A \times B) \to \mathcal{H}$ by $g(\alpha, \beta) = \mu(\alpha \otimes \Id_M \otimes \beta)$ (for all $(\alpha, \beta) \in \aut(A \times B)$). Then: 
\begin{equation} \label{isoaut0}
\aut_0(\T) \cong \GL(M) \times_\mathcal{H} \aut(A \times B),
\end{equation} 
that is, $\aut_0(\T)$ is the pullback in $\set$ of the following pullback square:
\begin{equation} \label{pullaut0}
\xygraph{ !{<0cm,0cm>;<1.5cm,0cm>:<0cm,1.2cm>::} 
!{(.2,-.3)}*+{\lrcorner}
!{(0,0) }*+{GL(M)\times_{\mathcal{H}} \aut(A\times B)}="a"
!{(3,0) }*+{\aut(A\times B)}="b" 
!{(0,-1.3)}*+{\GL(M)}="c" 
!{(3,-1.3)}*+{\mathcal{H}}="d"
!{(1.8,.1 )}*+{^{\pi_2}}="u"
"a":"b" "a":_{\pi_1}"c"
"b":^g"d" "c":_f"d"
}
\end{equation}
\end{theorem}

\begin{proof}
We identify $\aut(A \times B)$ with $\aut(A) \times \aut(B)$. In $\set$, consider the pullback 
\[
\GL(M) {\times}_\mathcal{H} \aut(A \times B) {=} \big \{ (T, (\alpha, \beta)) | \, \, T\mu (\Id_A \otimes T^{-1} \otimes \Id_B) {=} \mu(\alpha \otimes \Id_M \otimes \beta) \big \} 
\]
Apply Lemma \ref{req} to \eqref{pullaut0} to obtain that $\GL(M) {\times}_\mathcal{H} \aut(A \times B)$ is a group. We claim that $\GL(M) {\times}_\mathcal{H} \aut(A \times B)$ is isomorphic to $\GL_{A, B}(M)$. In fact, notice
that 
$(T,(\a,\b))\in \GL(M) {\times}_\mathcal{H} \aut(A \times B)$
if and only if $T(amb) = \alpha(a) T(m) \beta(b)$ for all $a \in A$, $m \in M$ and $b \in B$. The map $\GL_{A, B}(M) \to \GL(M) {\times}_\mathcal{H} \aut(A \times B)$ given by $T \mapsto (T,(\alpha_T, \beta_T))$, for $T \in \GL_{A, B}(M)$, is the desired  isomorphism. The result now follows from Corollary \ref{aut0}.
\end{proof}

\begin{corollary}
Let $\T = \Tri(A, M, B)$ be triangular algebra with $M$ faithful. Then $\der_0(\T) \cong \gl(M) \times_{\mathcal{H}} \der(A \times B)$, where $\gl(M)$ refers to the $R$-module consisting on all $R$-endomorphims of $M$ and $\mathcal{H}$ is like in Theorem \ref{findingaut0}. 
\end{corollary}

\begin{proof}
For $\A \in \alg_R$, let $\affaut_0(\T)(\A)$ be the set of elements $\sigma \in \aut(\A_R)$ such that $\sigma(p \otimes 1) = p \otimes  1$. It is straightforward to check that $\affaut_0(\T)(\A) \in \grp$, and so $\affaut_0(\T)$ is a subfunctor of $\affaut(\T)$. 
Apply Proposition \ref{squareprop} and Theorem \ref{findingaut0} to conclude the proof.
\end{proof}

\begin{remark}
The inner automorphisms of $\aut_0(\T)$, via the isomorphism \eqref{isoaut0}, correspond to pairs $\big(T,(\inn(x),\inn(y))\big)$, for some $x\in A^\times$ and $y\in B^\times$ such that $T(m) = xmy^{-1}$ for any $m\in M$. 
\end{remark}

The group $\G$ consisting of all $T\in\GL(M)$ of the form $T(m) = xmy$, for $x\in A^\times$, $y\in B^\times$ is indeed isomorphic
to a factor group of $A^\times\times (B^\times)^{\hbox{\tiny op}}$. To show this, let $E:=\End_R(M)$ and consider the pullback square 
$$\xygraph{ !{<0cm,0cm>;<1cm,0cm>:<0cm,1cm>::} 
!{(.2,-.4)}*+{\lrcorner}
!{(0,0) }*+{A^\times\times_E (B^\times)^{\hbox{\tiny op}}}="a" 
!{(2.3,0) }*+{(B^\times)^{\hbox{\tiny op}}}="b" 
!{(0,-1.7) }*+{A^\times}="c" 
!{(2.3,-1.7)}*+{E}="d" 
"a":"b" "a":"c"
"b":^R"d"
"c":_{L}"d"}$$
where $L(a)(m) = am$ and $R(b)(m) = mb$ for all $m \in M$. Assume that $M$ is faithful, then $L_xR_{y^{-1}}= \mathrm{Id}_M$ implies $x \in Z(A)$ and $y \in Z(B)$. This proves that 
$$
A^\times\times_E(B^\times)^{\hbox{\tiny op}} = Z(A)^\times\times_E Z(B^\times)^{\hbox{\tiny op}},
$$
and $\G \cong (A^\times\times(B^\times)^{\hbox{\tiny op}})/(Z(A^\times)\times_E Z(B^\times)^{\hbox{\tiny op}})$, as desired.   
Moreover, the map

\noindent $\Phi: A^\times\times (B^\times)^{\hbox{\tiny op}}\to \innaut_0(\T)$ given by $\Phi(x,y) = \big(L_xR_{y^{-1}},(\inn(x),\inn(y))\big)$ is a group epimorphism with kernel $Z(A^\times)\times_E Z(B^\times)^{\hbox{\tiny op}}$. Thus, $\G \cong\innaut_0(\T)$.

%

Using group functors, this result gives us the existence of an 
exact sequence 
\begin{equation}\label{sequence}
1\to \affGL\nolimits_1(A)\times_E\affGL\nolimits_1(B^{\hbox{\tiny op}})\to 
\affGL\nolimits_1(A)\times\affGL\nolimits_1(B^{\hbox{\tiny op}})\to\affinnaut\nolimits_0(\T)\to 1.
\end{equation}

Our last goal is to describe $\innder_0(\T)$ in terms of a certain pullback in $\set$.

\begin{lemma}
The center $Z(\T)$ of a triangular algebra $\T = \Tri(A, M, B)$ is the pullback of the following pullback square:
$$
\xygraph{ !{<0cm,0cm>;<1cm,0cm>:<0cm,1cm>::} 
!{(.2,-.4)}*+{\lrcorner}
!{(0,0) }*+{Z(A)\times_E Z(B)}="a" 
!{(2.5,0) }*+{Z(B)}="b" 
!{(0,-1.7) }*+{Z(A)}="c" 
!{(2.5,-1.7)}*+{E}="d" 
"a":"b" "a":"c"
"b":^R"d"
"c":_{L}"d"
}
$$
where $Z(A)$ and $Z(B)$ are the centers of $A$ and $B$, respectively, $E:= \End_R(M)$ and $L: Z(A) \to E$ and $R: Z(B) \to E$ are given by $L(a) := L_a$ and $R(b) := R_b$, for all $a \in A$, $b \in B$. 
\end{lemma}

\begin{proof}
From \cite[Proposition 3]{Ch2}, we have that 
\[
\Z(\T) = \left\{
\left(
\begin{array}{@{}cc@{}}
a & 0  \\
  & b
\end{array}
\right) \in \T  \,\middle|\, a \in \Z(A), \, b \in \Z(B) \, \mbox{ and } \, am = mb, \,
\forall \, m\in M
\right\}.
\]
In terms of the multiplication maps $L$ and $R$, we have that 
\[
\Z(\T) = \left\{
\left(
\begin{array}{@{}cc@{}}
a & 0  \\
  & b
\end{array}
\right) \in \T  \,\middle|\, a \in \Z(A), \, b \in \Z(B) \, \mbox{ and } \, L(a) = R(b)
\right\},
\]
and so $Z(\T) \cong Z(A) \times_E Z(B)$.   
\end{proof}

\begin{notation}\label{never}
Let $\A$ be an algebra. We write $\A^-$ to denote the resulting Lie algebra with the same underlying vector space as $\A$ but product $[x, y] = xy - yx$.
\end{notation}

\begin{proposition}
Let $\T = \Tri(A, M, B)$ be a triangular algebra. Then 
\[
\innder_0(\T) \cong (A^- \times B^-)/(Z(A) \times_E Z(B)) 
\]
\end{proposition}

\begin{proof}
Let $\ad: \T \to \der(\T)$ be the adjoint map. It is straightforward to check that $\innder_0(\T) = \ad \left(
\begin{array}{@{}cc@{}}
A & 0  \\
  & B
\end{array}
\right)$. Thus:
\[
\innder_0(\T) \cong A^- \times B^- /(Z(A) \times_E Z(B)) 
.\]
Alternatively, we could prove this result applying the Lie functor to \eqref{sequence} and using Proposition \ref{squareprop}. This would
gives us the following exact sequence
$$0\to Z(A)\times_E Z(B)\to A^-\times B^-\to\innder_0(\T)\to 0,$$
and the proof follows. 
\end{proof}

\section*{Acknowledgements}
The first and second authors are supported by the Junta de Andaluc\'{\i}a and Fondos FEDER, jointly, through projects  FQM-336 and FQM-7156 and also by the Spanish Ministerio de Econom\'ia y Competitividad and Fondos FEDER, jointly, through project MTM2016-76327-C3-1-P.
The third and fourth authors were supported by a CSUR grant from the NRF,
the National Research Foundation of South Africa. The third author thanks the Departmento de \'Algebra, Geometr\'ia y Topolog\'ia de la Universidad de M\'alaga (Spain) for its hospitality during her visit from June to July 2017.

\end{document}